\documentclass[12pt]{article}
\usepackage{amssymb}
\usepackage{amsmath}
\usepackage{amsfonts}
\usepackage{amsthm}
\usepackage{tcolorbox}
\newtheorem{theorem}{Theorem}
\newtheorem{pro}{Proposition}
\newtheorem{lemma}{Lemma}

\newtheorem{rem}{Remark}

\def\R{\mathbb{R}}
\def\N{\mathbb{N}}

\def\sgn{\mathop{\rm sgn}\nolimits}

\def\S{\mathcal S} 

\newcommand{\indi}[1] {1\hspace{-1mm} \mbox{\rm I}_{#1} }
\newcommand{\beq}{\begin{equation}}
\newcommand{\eneq}{\end{equation}}
\newcommand{\beqn}{\begin{eqnarray}}
\newcommand{\eneqn}{\end{eqnarray}}
\newcommand{\beqnon}{\begin{eqnarray*}}
\newcommand{\eneqnon}{\end{eqnarray*}}

\newcommand{\re}[1]{(\ref{#1})}
\def\max{\mathop{\rm max}\nolimits}

\title{Self-similar solutions for the generalized fractional 
Korteweg-de Vries equation}

\author{Luc Molinet\thanks{Institut Denis Poisson, Universit\'e de Tours, Universit\'e d'Orl\'eans, CNRS, Parc  Grandmont, 37200 Tours, France} \and St\'ephane Vento\thanks{Université Sorbonne Paris Nord, Laboratoire Analyse, Géométrie et Applications,LAGA, CNRS, UMR 7539, F‐93430, Villetaneuse, France } \and Fred Weissler\footnotemark[2] }

\begin{document}
\maketitle

\abstract{We consider the Cauchy problem for the generalized fractional 
Korteweg-de Vries equation
$$
u_t + D^\alpha u_x + u^p u_x = 0,\quad 1<\alpha\le 2,\quad p\in \N^*
$$
with homogeneous initial data $\Phi$. We show that, under smallness assumption on $\Phi$, and for a wide range of $(\alpha, p)$, including $p=3$, we can construct a self-similar solution of this problem.
}



\section{Introduction}
The aim of this paper is to study the existence of self-similar solutions
 to the generalized fractional KdV  equation
\begin{equation} \label{KdV}
\left\{
\begin{array}{l}
u_t + D^\alpha u_x + u^p u_x = 0 \\
u(0)=\Phi
\end{array}
\right.
\end{equation}
where $u=u(t,x)$ is real valued function  defined on $\R_+\times\R$  and
 $p\in \N^*$. Here $1< \alpha\le 2$ and $D^\alpha$ is the Fourier multiplier with symbol $|\xi|^\alpha$.
 
Equation \eqref{KdV} is invariant under the scaling 
$$
u(t,x) \mapsto u_\lambda(t,x) = \lambda^{\frac\alpha p} u(\lambda^{\alpha+1} t ,\lambda x) , \quad \forall
\lambda>0, \; (t,x)\in \R_+ \times \R .
$$
We look for self-similar solutions of \eqref{KdV} that are fixed points of the above scaling mapping, they are of the form
\begin{equation}\label{self}
u(t,x) = t^{-\frac{\alpha}{(\alpha+1)p}} v(x/t^\frac{1}{\alpha+1}) , \quad (t,x)\in \R_+^* \times
 \R
\end{equation}
where $v = u(1)$ is the profile of the self-similar solution $u$. We infer from above that
the initial data $u(0,x)=\Phi(x)$ must be an homogeneous
function of degree $-\alpha/p$ and take the form
\begin{equation}\label{Phi-form}
\Phi(x) = \omega(x) |x|^{-\frac{\alpha}p}, \quad \omega(x) = \left\{\begin{array}{ll}\omega_+, &x>0\\ \omega_-,& x<0\end{array}\right.
\end{equation}
for some $\omega_+,\omega_-\in \R$.

Inserting \eqref{self} into \eqref{KdV} we find that the profile $v$ solves the equation
\begin{equation}\label{profile-eq}
(D^\alpha v+\frac{1}{p+1}v^{p+1})' - \frac{1}{\alpha+1}(\frac \alpha p v + xv') = 0.
\end{equation}

The class of equations \eqref{KdV} contains in particular the (generalized) KdV equation ($\alpha=2$) and the (generalized) Benjamin-Ono equation ($\alpha=1$). There is a wide literature around self-similar solutions to the gKdV equation especially concerning the mKdV equation ($p=2$) for which the self-similar profile satisfies a Painlev\'e type equation. The existence and precise asymptotic of solutions  to this equation has been widely studied (see \cite{HL}, \cite{DZ}). 
The importance of self-similar profiles in the studying of the PDE's was underlying in \cite{HN} where it was proven that such profiles appear naturally in the long time asymptotic of solutions to the mKdV equation (see also \cite{GPR}, \cite{H}). Recently new interest on self-similar solutions to the mkdV equation appears due to the link with the behavior of vortex filaments in fluid dynamics. In a series of works \cite{CCV1}, \cite{CCV2}, \cite{CC1}  the local well-posedness for small subcritical perturbations of self-similar solutions to the mKdV equation is proven.  
In the critical and super critical case $ p\ge 4$, self-similar profiles were constructed with precise asymptotic in \cite{BW} by means of the associated ODE. However the associated solutions do not belong to $\dot{H}^1(\R)$. In contrast,    self-similar solutions with finite energy are constructed in \cite{K} for $ p$ slightly larger than $ 4$. 
Let us  also mention \cite{MR} were the existence of self-similar solutions in the case $ p\ge 4 $ is established by proving the global well-posedness for small data in the critical homogenous Besov space $ \dot{ B}^{s_p,\infty}_2 $, $ s_p=\frac{1}{2}-\frac{2}{p} $,   that contains the homogeneous functions with suitable degree. We notice that, up to our knowledge, no result are known in the case $ p=3 $.
Concerning the case $ \alpha<2$ where the equation \eqref{profile-eq} for the self-similar profile becomes non local, in particular the generalized Benjamin-Ono equation, we are not aware of any results in the literature. \\

Our aim here  is to propose  a different approach to  prove the existence of self-similar solutions to \eqref{KdV}. Our approach is based on the integral equation associated to \eqref{KdV} and heavily makes use of the Strichartz estimates associated with the  linear  group. Substituting the self-similar solution by its form \eqref{self} in the integral equation, a crucial point is  that the space derivative on the nonlinear term that falls on the profile in the Duhamel part can be dropped by integration by parts with respect to the temporal variable. Then the self-similar profile is obtained as a fixed point of a mapping linked with the resulting integral form that does not contain any derivative.

 We do not make use directly of  \eqref{profile-eq} and  we are able to  tackle the case $ \alpha\in ]1,2] $. Actually $ \alpha=1 $ seems to be a limit case for our approach and we do not succeed to take it into account. The main new features  that we obtained in this work are the existence of self-similar solutions for the generalized KdV equation  for  $ p= 3 $ and for the fractional generalized KdV equation \eqref{KdV} with $ \alpha\in ]1,2[ $, $p\ge 3 $ and   $ p>\frac{\alpha}{\alpha-1}$.

Our main results read as follows.
\begin{theorem}\label{main-theorem}
Let $1<\alpha\le 2$ and  $p\ge 3$, $p \in \N$, with  $p>\alpha/(\alpha-1)$, and let   $\Phi$ be  an homogeneous function of degree $-\alpha/p$  as in (\ref{Phi-form}), with $|A_\alpha(1)\Phi|_{L^{p+1}}$ small enough, where $A_\alpha(1)$ is defined below in \eqref{linear-group}.
Then there exists a  profile $v\in C_b(\R) $ such that  $ u(t,x) = t^{-\frac{\alpha}{(\alpha+1)p}} v(x/t^\frac{1}{\alpha+1}) $ is a self-similar solution of  IVP  \re{KdV}. More precisely, $u(t,x)$ satisfies the Duhamel formulation of \re{KdV} (see \re{integralKdV}) for any $x\in\R$ and $t>0$, and $u(t)\to \Phi$ in $\mathcal{S}'(\R)$ as $t\to 0$.  
\end{theorem}
With a supplementary constraint on $ p $ we are also able to prove that the self-similar profile $ v$ constructed above satisfies   \eqref{profile-eq}.
\begin{theorem}\label{theo2}
   If $p > 2\alpha /(\alpha-1)$, then the profile $v$ obtained in Theorem \ref{main-theorem} satisfies equation \eqref{profile-eq} in distributional sense and belongs to $C^\infty(\R) \cap W^{1,\infty}(\R) $.
   
   \end{theorem}
   
We conclude this introduction with a remark on the asymptotic behavior of the profile. Formally at least, we get that if $u$ is a self-similar solution of \re{KdV} with a profile $v$, then
$$
\Phi(x) = \lim_{t\searrow 0} u(t,x) = \lim_{t\searrow 0} t^{-\frac{\alpha}{(\alpha+1}p} v\left(\frac{x}{t^{1/(\alpha+1)}}\right) = |x|^{-\frac{\alpha}p} \lim_{r\to \infty} r^{\frac{\alpha}p}v\left(\sgn(x)r\right). 
$$  
In view of \re{Phi-form},  this forces 
$$
\lim_{r\to \infty} r^{\frac{\alpha}p}v\left(\sgn(x)r\right) = \omega(x).
$$   
Therefore, taking $x > 0$ it follows that
$$
v(r) \underset{+\infty}{\sim} \frac{\omega_+}{r^{\frac \alpha p}} \textrm{ if } \omega_+\neq 0 \quad \textrm{ and }\quad v(r) = \underset{+\infty}{o}\left(\frac{1}{r^{\frac{\alpha}p}}\right) \textrm{ if } \omega_+ = 0.
$$
A similar behavior may be obtained at $-\infty$ depending on $\omega_-$. In particular it seems that, for $ \Phi  $ non identically vanishing,  no  profile  belong to $L^q(\R) $ for $ 1\le q \le p/\alpha $. 

In order to establish  these various asymptotic behaviors rigorously, not just formally,
 one would need to show that initial value $\Phi$ is attained strongly enough as $t \to 0$ to imply that 
 indeed $\lim_{t\searrow 0} u(t,x) = \Phi(x)$ holds pointwise for all $x \neq 0.$
 By homogeniety, it would suffice to establish this limit for $x = \pm 1$.
\section{Notation and linear estimates}
\subsection{Notation}
For any positive numbers $a$ and $b$, the notation $a\lesssim b$ means that there exists a positive constant $c$ such that $a\le cb$. By $a\sim b$ we mean that $a\lesssim b$ and $b\lesssim a$.
Moreover, if $\gamma\in\R$, $\gamma^+$, respectively $\gamma^-$, will denote a number slightly greater, respectively lesser, than $\gamma$.

For any $1\le p\le \infty$, we denote by $|\cdot|_p$ the usual norm of the Lebesgue space $L^p(\R)$. The conjugate exponent of $p$ will be written $\overline{p}$, so that $1/p+1/\overline{p}=1$.
Moreover, if $s\in\R$, the Sobolev space $W^{s,p}(\R)$ is endowed with the norm
$$
|f|_{W^{s,p}} = |(1-\partial_{xx})^{s/2}f|_p .
$$ 
If $0<s<1$ and $1\le p\le\infty$, we also consider the Lipschitz space $\Lambda_s^{p,\infty}(\R)$ defined by its norm
$$
|f|_{\Lambda_s^{p,\infty}} = |f|_p + \sup_{t\neq 0} \frac{|f(\cdot + t) - f(\cdot)|_p}{|t|^s} ,
$$
and $C^{0,s} = \Lambda_s^{\infty,\infty}$ holds for the  H\"{o}lder space. Recall the classical Sobolev embeddings (see \cite{St}, Theorem 5 page 155)
\begin{equation}\label{lip-embedding}
  W^{s,p}(\R) \hookrightarrow \Lambda_s^{p,\infty}(\R),\quad 0<s<1, \quad p\ge 2
\end{equation} 
and
\begin{equation}\label{lip-embedding2}
   \Lambda_s^{p,\infty}(\R) \hookrightarrow W^{s,p}(\R),\quad 0<s<1, \quad p\le 2.
\end{equation}
When $1<s<2$, we say that $f\in \Lambda^{p,\infty}_s(\R)$ if and only if $f\in L^p(\R)$ and $f'\in \Lambda^{p,\infty}_{s-1}(\R)$. In this case \eqref{lip-embedding} and \eqref{lip-embedding2} still hold with $C^{0,s}=\Lambda^{\infty,\infty}_s$ again.

Let $A_\alpha(t)$ be the linear group of the generalized fractional 
Korteweg-de Vries equation,
\begin{equation}\label{linear-group}
  A_\alpha(t) = e^{tD^\alpha\partial_x}.
\end{equation}
 Then the integral
 equation associated with \re{KdV} reads
\begin{equation} \label{integralKdV}
u(t)=A_\alpha(t) \Phi -\int_0^t A_\alpha(t-t') (u^p u_x)(t')dt'.
\end{equation}

\subsection{Linear estimates}

The following Lemma will be of constant use.

\begin{lemma}\label{Strichartzestimates}
\begin{enumerate}
\item  For any $ r\in [1,2] $, there exists $ C>0 $ such that
\begin{equation}
| A_\alpha(1) \varphi |_{W^{\frac{\alpha-1}2(\frac 2r-1), \overline r}}
\leq C |\varphi |_r, \quad \forall \varphi
\in L^r. \label{strichartz2}
\end{equation}
where $ 1/r+1/{\overline r} =1 $.
\item For any $q\in[2,\infty]$, there exists $C>0$ such that
\begin{equation} \label{strichartz}
|A_\alpha(1)\varphi|_q\le C |\varphi|_r,\quad \forall r\in\left[\max\left(1, (\frac{\alpha+1}2-\frac \alpha q)^{-1}\right), {\overline q}\right],
\end{equation}
provided the right-hand side is finite.

\end{enumerate}
\end{lemma}
\begin{proof}
Estimate \re{strichartz2} can be found in (\cite{KPV}, Corollary 2.3 page 330). To obtain \re{strichartz}, we first apply \re{strichartz2} with $r={\overline q}$ to get for all $\varphi$ that
$$
|A_\alpha(1)\varphi|_q\le C|\varphi|_{W^{-\frac{\alpha-1}2(1-\frac 2q), {\overline q}}}.
$$
Thus by Sobolev embeddings, we are done if $1\le r\le {\overline q}$ and $1/r-1/{\overline q}\le \frac{\alpha-1}2(1-2/q)$, which is the condition stated in \re{strichartz}. Note that the interval in \re{strichartz} is never empty since 
$$
(\frac{\alpha+1}{2}-\frac{\alpha}{q})^{-1} \le \overline{q} \Leftrightarrow \frac{\alpha+1}{2}-\frac{\alpha}{q}\ge \frac{1}{\overline{q}}=1-\frac{1}{q} \Leftrightarrow  \alpha( \frac{1}{2}-\frac{1}{q})\ge \frac{1}{2}-\frac{1}{q} 
$$
that is satisfied for $ \alpha\ge 1 $ and $ q\ge 2 $. 
\end{proof}

Now we state a result concerning the regularity of the linear part in our integral equation \re{integralKdV}.

 \begin{lemma} \label{linearlemma}
If $p\ge 3$ and $ \Phi $ is an homogeneous function of degree $ -\alpha/p$, then
  \begin{equation}
 A_\alpha(1)\Phi \in L^{\max(\frac{p}{\alpha}, (\frac{1}{2} -\frac{1}{p}+\frac{1}{2\alpha})^{-1})^+}\cap W^{(\frac{\alpha+1}2-\frac\alpha p)^-,\infty} \: .
 \label{linear}
 \end{equation}
 \end{lemma}
 \proof
 Let $ \hat{\chi} \in C^\infty_0 (\R) , \; \hat{\chi}\equiv 1 $
  on a neighborhood of the origin. We split $ \Phi $ in a
   low frequencies part and a high frequencies part
  by setting
  $$
  \Phi = \Phi_1 + \Phi_2 \quad \hbox{ with } \Phi_1=\chi\star \Phi
  \;.
  $$
  To treat the low frequencies we note that
   $ A_\alpha(1)\Phi_1=g\star \Phi $ with $ g=A_\alpha(1)\chi $. Since
   $ |\Phi (x) |\leq c/|x|^{\alpha/p} $ by \re{Phi-form} and for any $ k\in \N $
   $$
   \widehat{D^k g}=|\xi|^k \, e^{i\xi|\xi|^\alpha} \hat{\chi}(\xi)\in {\cal
   S}(\R) ,
   $$
   it is easy to see that
   $$
   |D^k A_\alpha(1)\Phi_1 |=|D^k g \star \Phi |\leq
   \frac{c}{(1+|x|)^{\alpha/p}}.
  $$
 Hence, for any $ k\in \N $,
 \beq
 D^k A_\alpha(1) \Phi_1 \in L^q ,\quad p/\alpha< q \leq \infty \; .
 \label{lowfreq}
 \eneq
On the other hand, we get from \cite{St}, Lemma page 73 that for any $ 0<\theta< 1 $, 
$$
{\cal F}\left(\frac{\sgn(x)}{|x|^{\theta}}\right)(\xi) = c\frac{\sgn(\xi)}{|\xi|^{1-\theta}},
$$
and it follows that
$$
\hat{\Phi}_2(\xi)=\frac{1-\hat{\chi}(\xi)}{|\xi|^{1-\alpha/p}}(c_1+c_2\sgn(\xi)) \in
C^\infty .
$$
Therefore, at infinity $ \Phi_2 $ and its derivatives decrease faster than
any polynomial. It thus suffices to study the local
behavior of $ \Phi_2 $. From above, one has for $ 0\le \beta< 1-\alpha/p $,
\begin{align*}
D^\beta \Phi_2(x)&={\cal
F}^{-1}\Bigl(\frac{c_1+c_2\sgn(\xi)}{|\xi|^{1-\alpha/p-\beta}}\Bigr) - {\cal
F}^{-1}\Bigl(\frac{\hat{\chi}(\xi)(c_1+c_2\sgn(\xi))}{|\xi|^{1-\alpha/p-\beta}}\Bigr)\\
&=\frac{c_1'+c_2'\sgn(x)}{|x|^{\alpha/p+\beta}}- a(x)
\end{align*}
where $ a\in C^\infty(\R) $. \\
In particular, $ \Phi_2\in L^q(\R), \; 1\leq q < p/\alpha $, and
 $ D^\beta \Phi_2\in L^1(\R) $ whenever   $ 0\le \beta< 1-\alpha/p $.
It then follows from \eqref{strichartz2} with $r=1 $ that 
 $ D^\beta A_\alpha(1)\Phi_2 \in W^{\frac{\alpha-1}{2},\infty} $ for   $0\le  \beta< 1-\alpha/p $ that is 
 $ A_\alpha(1)\Phi_2 \in W^{(\frac{\alpha+1}{2}-\frac{\alpha}{p})^-,\infty} $.
 Moreover, according to  \eqref{strichartz}, $ A_\alpha(1)\Phi_2\in L^q(\R) $ if and only if 
 $$
 [1,\frac{p}{\alpha}[ \cap \left[\max\left(1, (\frac{\alpha+1}2-\frac \alpha q)^{-1}\right), {\overline q}\right]\neq \emptyset
 $$
 that is 
 $$
 \frac{p}{\alpha} >(\frac{\alpha+1}{2}-\frac{\alpha}{q})^{-1} \Leftrightarrow \frac{\alpha}{p} < \frac{\alpha+1}{2} -\frac{\alpha}{q} \Leftrightarrow \frac{1}{q} <\frac{1}{2} -\frac{1}{p}+\frac{1}{2\alpha} \; .
 $$
 We thus infer that 
\beq
 A_\alpha(1) \Phi_2\in L^{q^+}\cap W^{(\frac{\alpha+1}2-\frac\alpha p)^-,\infty} , \quad q>(\frac{1}{2} -\frac{1}{p}+\frac{1}{2\alpha})^{-1} \; .
 \label{highfreq}
\eneq
 Gathering \re{lowfreq} and \re{highfreq}, one obtains the desired
 result. \hfill
 $ \Box $

\section{Construction of the profile}

The profile of the self-similar solution will be obtained as a fixed point of a
certain mapping relating to the integral equation \re{integralKdV}. To
understand how we construct this mapping, assume that $u$ is a
self-similar solution of \re{integralKdV}. Then one has
\begin{align}
u(t,x) & = \lambda^{\frac\alpha p} u(\lambda^{\alpha+1} t ,\lambda x) , \quad \forall
\lambda>0, \; (t,x)\in \R_+ \times \R  \nonumber  \\
 & = t^{-\frac{\alpha}{(\alpha+1)p}} u(1,x/t^{\frac 1{\alpha+1}}) , \quad (t,x)\in \R_+^* \times
 \R.
 \label{selfsimil}
\end{align}
The function $v(\cdot)=u(1,\cdot)$ is called the profile of the
self-similar solution $u$ and  one infers from above that
the initial data $u(0,x)=\Phi(x)$ must be an homogeneous
function of degree $-\alpha/p$.
 Setting $g=v^{p+1}/(p+1)$, and using that  by dilation symmetry, for all $\lambda>0 $, $ s\in\R $  and all $\varphi\in {\mathcal S}'(\R) $  it holds 
 \begin{equation}\label{tata}
 (A_\alpha(\lambda^{\alpha+1}s)\varphi)(\lambda\cdot)=(A_\alpha(s) \varphi(\lambda\cdot))(\cdot) \; ,
 \end{equation}
one gets at least formally
\begin{align}
 u(t,x) &=  (A_\alpha(t)\Phi)(x) -\left[A_\alpha(t) \int_0^t A_\alpha(-\tau)
(u^p(\tau)
u_x(\tau) ) d\tau\right] (x) \nonumber\\
 & =  (A_\alpha(t)\Phi)(x) -\left[A_\alpha(t) \int_0^t \tau^{-1-\frac\alpha{(\alpha+1)p}} \left(A_\alpha(-\tau)
 g'\left(\frac{\cdot}{\tau^{\frac 1{\alpha+1}}} \right)\right) d\tau\right](x) \nonumber \\
  & =  (A_\alpha(t)\Phi)(x) -\left[A_\alpha(t) \int_0^t \tau^{-1 -\frac\alpha{(\alpha+1)p} }\left(A_\alpha(-1)
 g' \right) \left(\frac{\cdot}{\tau^{\frac 1{\alpha+1}}} \right) d\tau \right](x) \nonumber\\
 &=  (A_\alpha(t)\Phi)(x) -(A_\alpha(t) G(t,\cdot))(x).\label{toto}
\end{align}
with 
$$
 G(t,y):= \int_0^t \tau^{-1 -\frac\alpha{(\alpha+1)p} }\left(A_\alpha(-1)
 g' \right) \left(\frac{y}{\tau^{\frac 1{\alpha+1}}} \right) d\tau .
$$
 Performing the change of variable $\sigma \mapsto |y|/\tau^{\frac 1{\alpha+1}}$ in the above integral,  using 
 again \eqref{tata}
 and that $A_\alpha(t)$ commutes with spatial derivatives, we
 obtain for $t>0$,
 \begin{align*}
 G(t,y)&= \frac{\alpha+1}{|y|^{\frac \alpha p}}
 \int^\infty_{|y|/t^{\frac 1 {\alpha+1}}} \left( A_\alpha(-1) g' \right)
 (\sigma \sgn(y)) \sigma^{-1+\frac \alpha p} d\sigma \\ 
 &= \frac{\alpha+1}{|y|^{\frac \alpha p}}
 \int^\infty_{|y|/t^{\frac 1 {\alpha+1}}} \left( A_\alpha(-1) g\right)'
 (\sigma \sgn(y)) \sigma^{-1+\frac \alpha p} d\sigma
\end{align*}
 Setting $t=1$ and integrating by parts, using that $ p>\alpha\Rightarrow 1-\frac{\alpha}{p}>0$, it leads to 
  \begin{align*}
 G(1,y)&= (\alpha+1)\frac{\sgn(y)}{|y|^{\alpha/p}} \,\Bigl(- (A_\alpha(-1)g)(y) \, |y|^{-1+\frac{\alpha}{p}} \\
 &\quad +(1-\frac{\alpha}{p})\int^\infty_{|y|} (A_\alpha(-1)g)(\sigma \,
  \sgn(y))\sigma^{-2+\alpha/p}
 \, d\sigma \Bigr) \\
&= -(\alpha+1) (1-\frac{\alpha}{p})\frac{\sgn(y)}{|y|^{\alpha/p}}\,\int^\infty_{|y|}
 \sigma^{-2+\alpha/p}\\
 &\quad \times \Bigl(
 (A_\alpha(-1)g)(y) 
   - (A_\alpha(-1)g)(\sigma \, \sgn(y))\Bigl) \, d\sigma \, .
\end{align*}
Therefore  \eqref{toto}  with $ t=1$ gives
 \begin{align}\label{eqint-u1}
 u(1)   &=  A_\alpha(1)\Phi +(\alpha+1)(1-\frac{\alpha}{p})\,  A_\alpha(1)\Bigl[
\frac{\sgn(\cdot)}{|\cdot|^{\alpha/p}}\,\int^\infty_{|\cdot|}
 \sigma^{-2+\alpha/p} \Bigl(
 (A_\alpha(-1)g)(\cdot) \nonumber \\
  & \hspace{60mm} - (A_\alpha(-1)g)(\sigma \, \sgn(\cdot))\Bigl) \, d\sigma \, \Bigr]
 \; .
 \end{align}
 We denote by $ F=F_\Phi $ the map
 \beqnon
 v\in L^{p+1} & \longmapsto & A_\alpha(1)\Phi
 +(\alpha+1)(1-\frac\alpha p)\,  A_\alpha(1)\Bigl[
\frac{\sgn(\cdot)}{|\cdot|^{\alpha/p}} \,\int^\infty_{|\cdot|}
 \sigma^{-2+\alpha/p} \Bigl(
 (A_\alpha(-1)g)(\cdot)\\
 &&  \hspace{60mm}-(A_\alpha(-1)g)(\sigma \, \sgn(\cdot))\Bigl) \, d\sigma \,
 \Bigr]\; ,
 \eneqnon
 where $ g=v^{p+1}/(p+1) $.  \\
 Let us show that for $ | A_\alpha(1) \Phi |_{p+1} $ small enough, $ F_\Phi $
 is strictly contractive in a ball of $ L^{p+1} $.

\begin{pro} \label{fixedpointtheorem}
Let $1<\alpha\le 2$ and $p>\max(2, 1/(\alpha-1))$.
For $ | A_\alpha(1) \Phi |_{p+1} $ small enough, there exists a unique fixed point $v$ of $ F $  in a ball of $ L^{p+1} $.
\end{pro}

\begin{rem}
  The condition on $(\alpha, p)$ is less restrictive than in Theorem \ref{main-theorem}. Stronger assumptions will be required to rigorously justify the calculations that lead to \eqref{eqint-u1}, see Proposition \ref{regularitytheorem}. 
\end{rem}

\proof Due to Lemma \ref{linearlemma}, $ A_\alpha(1)\Phi \in L^{\max(\frac{p}{\alpha}, (\frac{1}{2} -\frac{1}{p}+\frac{1}{2\alpha})^{-1})^+}\cap L^\infty$. Obviously, 
 $\frac{p}{\alpha}<p+1 $ and observe that 
 $$
  (\frac{1}{2} -\frac{1}{p}+\frac{1}{2\alpha})^{-1}<p+1 \Leftrightarrow \frac{1}{2} -\frac{1}{p}-\frac{1}{p+1}+\frac{1}{2\alpha}>0
  $$
  and that  $ h\,:\, p \mapsto \frac{1}{2} -\frac{1}{p}-\frac{1}{p+1}+\frac{1}{2\alpha} $ is an increasing function on $ \R_+^* $ with $ h(3)= -\frac{1}{12} +\frac{1}{2\alpha} >0 $ for $\alpha<6 $. 
  Therefore for $p\ge 3 $ and $ 1\le \alpha\le 2$, 
  $$
  A_\alpha(1) \Phi \in L^{p+1} 
  $$
and it remains to consider the
term $(\alpha+1)(1-\frac{\alpha}{p})A_\alpha(1)R$  where $R$ is defined by
 \begin{equation}\label{defR}
R(x)=\frac{\sgn(x)}{|x|^{\alpha/p}} \int^\infty_{|x|}
 \sigma^{-2+\alpha/p} \Bigl(
 (A_\alpha(-1)g)(x)- (A_\alpha(-1)g)(\sigma \sgn(x))\Bigl) \, d\sigma \;.
 \end{equation}
 Let $ v\in L^{p+1} $, then $ g=v^{p+1}/(p+1) \in
L^1 $ and according to Lemma \ref{Strichartzestimates},
\begin{equation}\label{toto}
f=A_\alpha(-1) g \in W^{\frac{\alpha-1}2,\infty}  \quad \text{with} \quad \|f\|_{W^{\frac{\alpha-1}2,\infty}}\lesssim \|g\|_{1} \; .
\end{equation}
In particular, by Sobolev embedding \eqref{lip-embedding},
 $$
 f\in C^{0,\frac{\alpha-1}2}(\R) .
 $$ 
  Let us show that
 \beq
 |R|_r \lesssim |g|_{1} , \quad \forall r\in
 \left]1,\min(\frac{2}{3-\alpha}, \frac{p}{\alpha})\right[.   \label{est1R}
 \eneq
 For any $0\le \delta\le \frac{\alpha-1}2$, we get
  \begin{align*}
| R(x) | &\leq  \frac{1}{|x|^{\alpha/p}}
 \int^\infty_{|x|} \sigma^{-2+\alpha/p} |f(\sigma \sgn(x))-f(x)|
 d\sigma \\
& \lesssim \frac{|f|_{C^{0, \frac{\alpha-1}2}}}{|x|^{\alpha/p}}
 \int^\infty_{|x|} \sigma^{-2+\alpha/p} \left|\sigma-|x|\right|^{\delta}
 d\sigma \\
 & \lesssim \frac{|f|_{C^{0, \frac{\alpha-1}2}}}{|x|^{\alpha/p}}
 \int^\infty_{|x|} \sigma^{-2+\alpha/p+\delta}
 d\sigma .
 \end{align*}
 Since according to \eqref{toto}, $ |f|_{C^{0,\frac{\alpha-1}2}} \leq c|g|_{1}$, for $\delta\le \frac{\alpha-1}2 $ with $ \delta <1-\frac \alpha p$  we infer
 $$
 |R(x)| \lesssim |g|_{1}|x|^{\delta-1} .
 $$
 Splitting $ R $ in the following way:
 $$
 R=R \indi{0<|x|<1}+ R \indi{|x|>1},
 $$
 it is deduced that 
 $$
 |R(x)| \lesssim |g|_{L^1} (|x|^{-\gamma} \indi{0<|x|<1} + |x|^{-1} \indi{|x|>1})  , \quad \gamma>\max(\frac{3-\alpha}{2}, \frac{\alpha}{p})
 $$
 and \eqref{est1R} follows directly. Now we observe  that 
 $$
 \frac{2}{3-\alpha}>(\frac{\alpha+1}{2}-\frac{\alpha}{p+1})^{-1} \Leftrightarrow p >\frac{1}{\alpha-1}
 $$
 and that 
 $$
 \frac{p}{\alpha}>(\frac{\alpha+1}{2}-\frac{\alpha}{p+1})^{-1} \Leftrightarrow \frac{\alpha+1}{2}-\frac{\alpha}{p+1}-\frac{p}{\alpha}>0 \; .
 $$
 This last inequality is always satisfied for $ p\ge 3 $ and $ \alpha\in ]1,2] $ since  $ h \, :\, p\mapsto \frac{\alpha+1}{2}-\frac{\alpha}{p+1}-\frac{\alpha}{p}$ is  an increasing function on $ \R_+^* $ 
  with $ h(3)= \frac{\alpha+1}{2}-\frac{7\alpha}{12} >0 $ for $\alpha<6 $.
 Therefore the assumptions  $p>\max(2,\frac{1}{\alpha-1})$ and  $1<\alpha\le 2 $ ensure that
 $$
 \left]1,\frac{2}{3-\alpha}\right[ \cap \left[\max\left(1, (\frac{\alpha+1}2-\frac \alpha {p+1})^{-1}\right), {\overline {p+1}}\right] \neq \emptyset .
 $$
  and  Lemma \ref{Strichartzestimates} and the definition of $ g$ then  lead to 
 \begin{equation}\label{estRR}
 | A_\alpha(1) R |_{p+1}\lesssim |g|_{1}\lesssim |v|^{p+1}_{p+1}  \; .
 \end{equation}
 This proves that 
  \beq
 |F(v)|_{p+1} \leq |A(1)\Phi |_{p+1} +c|v|^{p+1}_{p+1}.
\eneq
Now for $(v_1,v_2)\in (L^{p+1})^2 $ we notice that 
$$
F(v_1)-F(v_2) = (\alpha+1) (1-\frac{\alpha}{p}) A_\alpha(1)(R[g_1]-R[g_2]) 
$$
where for $ i=1,2$, $ g_i =v_i^{p+1}/(p+1)$ and for all $x\in\R $, 
$$
R[g] (x)=\frac{\sgn(x)}{|x|^{\alpha/p}} \int^\infty_{|x|}
 \sigma^{-2+\alpha/p} \Bigl(
 (A_\alpha(-1)g)(x)- (A_\alpha(-1)g)(\sigma \sgn(x))\Bigl) \, d\sigma \;.
$$
In view of the linearity of $R $ with respect to  $ g$ and \eqref{estRR} we infer that 
$$
|A_\alpha(1) (R[g_1]-R[g_2])|_{p+1} =  |A_\alpha(1) (R[g_1-g_2])|_{p+1}\lesssim |g_1-g_2|_1 
$$
that leads  to 
\beq
 |F(v_1)-F(v_2)|_{p+1} \leq c (|v_1|^{p}_{p+1}+|v_2|^{p}_{p+1})
 |v_1-v_2|_{p+1},
\eneq
since by H\"older inequality,
$$
|g_1-g_2|_1 =\frac{1}{p+1}|v_1^{p+1}-v_2^{p+1}|_{1} \lesssim (|v_1|^{p}_{p+1}+|v_2|^{p}_{p+1}) |v_1-v_2|_{p+1}
$$
 This completes the proof of the theorem.\hfill$ \Box $
 
 \section{Proof of Theorem \ref{main-theorem}}
 To show that a fixed point $v$ of $ F $ is the profile of a self-similar solution of the integral
 equation, we have to integrate by parts in $\sigma$. To justify this
 integration by parts we will prove that $ f=A_\alpha(-1)
 \frac{v^{p+1}}{p+1} $ belongs at least to $ C^1(\R) $. This is
 the aim of the following proposition.

 \begin{pro}\label{regularitytheorem}
Assume that  $1<\alpha\le 2$   and $p>\alpha/(\alpha-1)$. Let $ v\in L^{p+1} $ be a fixed point of $ F $, then
$$
f=A_\alpha(-1)\frac{v^{p+1}}{p+1}\in C^1(\R)\cap C^{0,1^+}(\R) .
$$
\end{pro}
The proof will be based on an iteration of the following lemma.

\begin{lemma}\label{regularitylemma}
 Under the assumptions of Proposition \ref{regularitytheorem}, for any $\beta\in]0, 1-\alpha/p[$ such that $f\in C^{0,\beta}$, we have $f\in C^{0,(\beta+\alpha-1)^-}$.
\end{lemma}

\begin{proof}
Recall that $ v= A_\alpha(1) \Phi+(\alpha+1)(1-\frac\alpha p)\, A_\alpha(1) R $ with $R$ defined in \eqref{defR}. 
First we claim that $ v\in L^{p^+}(\R) $. For the linear part, this follows directly from Lemma \ref{linearlemma} since $ \frac{p}{\alpha}< p $ and 
$$
\frac{1}{2}-\frac{1}{p}+\frac{1}{2\alpha}>\frac{1}{p} \Leftrightarrow \frac{1}{2}-\frac{2}{p}+\frac{1}{2\alpha}>0,
$$
which is always satisfied for $ p\ge 3 $ and $ \alpha <3$. 
On the other hand, estimates \re{strichartz} and \re{est1R} give
   $ A_\alpha(1) R \in L^{p^+}(\R)$ since
$$
\left[\max\left(1, (\frac{\alpha+1}2-\frac{\alpha}{p^+})^{-1}\right), \overline{p^+}\right]\cap \left]1, \frac{2}{3-\alpha}\right[\neq \emptyset\textrm{ for } p>\frac\alpha{\alpha-1}.
$$
Assume now $f=A_\alpha(-1)v^{p+1}\in C^{0,\beta}$ for some $0<\beta<1-\alpha/p$ and let us estimate the following quantity:
\arraycolsep1pt
 \beqn
  R(x+t)&-& R(x) =\frac{\sgn(x)}{|x|^{\alpha/p}}
 \int_{|x+t|}^{\infty} \sigma^{-2+\alpha/p}\Bigl( f(x+t)-f(x) \Bigr)\,
 d\sigma \nonumber \\
  && + \frac{\sgn(x)}{|x|^{\alpha/p}}
 \int_{|x+t|}^{\infty} \sigma^{-2+\alpha/p}\Bigl( f(\sigma\sgn(x)) - f(\sigma \, \sgn(x+t)) \Bigr)\,
 d\sigma \nonumber \\
   && + \frac{\sgn(x)}{|x|^{\alpha/p}}
 \int_{|x|}^{|x+t|} \sigma^{-2+\alpha/p}\Bigl( f(\sigma\, \sgn(x))
 -f(x) \Bigr)\,
 d\sigma \nonumber \\
 && + \Bigl(
\frac{\sgn(x+t)}{|x+t|^{\alpha/p}}- \frac{\sgn(x)}{|x|^{\alpha/p}} \Bigr)
\int_{|x+t|}^{\infty} \sigma^{-2+\alpha/p}
\Bigl( f(x+t) - f(\sigma\sgn(x+t))\Bigr) \, d\sigma \nonumber \\
 && = R_1(t,x) + R_2(t,x)+R_3(t,x)+R_4(t,x) \; .
\label{decR}
 \eneqn
\arraycolsep5pt
Let us estimate one by one the contribution of these four terms. Below the implicit constant will depend on $|f|_{C^{0,\frac{\alpha-1}{2}}}$. 
\beqnon
 |R_1(t,x) | &\leq & \frac{1}{|x|^{\alpha/p}} |x+t|^{-1+\alpha/p}
 |f(x+t)-f(x)| \\
  & \lesssim  & \frac{|t|^\beta}{|x|^{\alpha/p}\,  |x+t|^{1-\alpha/p}}.
\eneqnon
 Performing the change of variable $ y=x/ t $ in the
integral,
 we deduce that
\beq \label{R1}
|R_1(t,\cdot)|_{1^+}  \lesssim |t|^{1^--1 +\beta}\,
 \Bigl| \frac{|x+1|^{-1+\alpha/p}}{|x|^{\alpha/p}}\Bigr|_{1^+}
  \leq c \, t^{\beta^-}
\eneq
To estimate the contribution of $R_3 $  we notice that $ |\sigma \sgn(x)-x|\le |\sigma|  +|x| $ and thus since $ 0<\beta<1 $, 
\begin{align*}
|R_3(t,x)| & \lesssim  \frac{1}{|x|^{\frac{\alpha}{p}}} \int_{|x|}^{|x+t|} \sigma^{-2+\frac{\alpha}{p}} (\sigma^\beta + |x|^\beta)\, d\sigma  \\
 & \lesssim  \frac{||x+t|^{-1+\frac{\alpha}{p}+\beta}
 -|x|^{-1+\frac{\alpha}{p}+\beta}|}{|x|^{\alpha/p}} +  \frac{||x+t|^{-1+\frac{\alpha}{p}}
 -|x|^{-1+\frac{\alpha}{p}}|}{|x|^{\frac{\alpha}{p}-\beta}}
\end{align*}
that leads by the same change of variable to 
\beqn
 |R_3(t,\cdot)|_{1^+} &\lesssim &  |t|^{\beta^-}\Bigl(\Bigl| \frac{|x+1|^{-1+\frac{\alpha}{p}+\beta}
 -|x|^{-1+\frac{\alpha}{p}+\beta}}{|x|^{\frac{\alpha}{p}}} \Bigr|_{1^+} 
 + \Bigl| \frac{|x+1|^{-1+\frac{\alpha}{p}}
 -|x|^{-1+\frac{\alpha}{p}}}{|x|^{\frac{\alpha}{p}-\beta}} \Bigr|_{1^+} \Bigr)\nonumber \\
  \label{R3} & \lesssim  & c \,|t|^{\beta^-}
\eneqn
where we used the mean value theorem to
 see that the functions inside the $ L^{1^+} $-norms of the
  right-hand side member
 decrease as $ |x|^{-2+
 \beta} $ at $ \infty $. \\
 To estimate $ |R_2 | $, we notice that
 \beqnon
 |R_2(t,x)| & \leq & \frac{\indi{|x|\leq |t|}}{|x|^{\alpha/p}}
 \,\int_{|x+t|}^{\infty} \sigma^{-2+\alpha/p}\Bigl| f(\sigma \, \sgn(x+t))
 -f(\sigma \, \sgn(x)) \Bigr| \,
 d\sigma \\
  & \lesssim  &  \, \frac{\indi{|x|\leq |t|}}{|x|^{\alpha/p}}
  |x+t|^{-1+\alpha/p+\beta}
  \eneqnon
  Therefore, performing the same change of variable as above,
\beqn
  |R_2(t,\cdot)|_{1^+} &\lesssim & |t|^{\beta^-}\;
  \Bigl| \frac{\indi{|x|\leq 1}}{|x|^{\alpha/p}\,
  |x+1|^{1-\alpha/p-\beta}}\Bigr|_{1^+}
  \nonumber \\
  \label{R2}& \lesssim & |t|^{\beta^-}
\eneqn
It remains to estimate the contribution of  $R_4  $.
 \beqnon
  |R_4(t,x)| & \leq & \Bigl|
\frac{\sgn(x+t)}{|x+t|^{\alpha/p}}- \frac{\sgn(x)}{|x|^{\alpha/p}} \Bigr|
\int_{|x+t|}^{\infty} \sigma^{-2+\alpha/p} \Bigl|
f(\sigma\sgn(x+t))-f(x+t)\Bigr| \, d\sigma \\
 & \leq & \Bigl|
\frac{\sgn(x+t)}{|x+t|^{\alpha/p}}- \frac{\sgn(x)}{|x|^{\alpha/p}} \Bigr|
\int_{|x+t|}^\infty  \sigma^{-2+\alpha/p} |\sigma-|x+t||^\beta\, d\sigma
\\
 & \leq & \Bigl|
\frac{\sgn(x+t)}{|x+t|^{\alpha/p}}- \frac{\sgn(x)}{|x|^{\alpha/p}} \Bigr|
 |x+t|^{-1+\alpha/p+\beta}
\eneqnon
 Hence, performing again the change of variable $
y=x/ t $ in the integral, it follows that
\begin{align}
 |R_4(t,\cdot)|_{1^+} & \leq  |t|^{\beta^-} \left|
\left|\frac{\sgn(y+1)}{|y+1|^{\alpha/p}}- \frac{\sgn(y)}{|y|^{\alpha/p}}
 \right|
 |y+1|^{-1+\alpha/p+\beta }\right|_{1^+}
  \nonumber \\
  \label{R4} & \leq  c |t|^{\beta^-}
\end{align}
since $ \left|\sgn(y+1)|y+1|^{-\alpha/p}-\sgn(y)|y|^{-\alpha/p} \right| $
 decreases
as $ |y|^{-1-\alpha/p} $ at $ \infty $. \\
Gathering \re{decR}--\re{R1}--\re{R3}--\re{R2}--\re{R4}, we infer that
 \beq
 |R(x+t)-R(x)|_{1^+} \leq c \, |t|^{\beta^-}.
 \eneq
 Thus $R$ belongs to the Lipschitz space $\Lambda^{1^+,\infty}_{\beta^-} $ and classical
 embedding \eqref{lip-embedding2} implies that $ R \in W^{\beta^-,1^+}$. From Lemma \ref{Strichartzestimates} it follows that
 $$
 A_\alpha(1) R \in W^{(\beta+\frac{\alpha-1}2)^-, \infty^-}.
 $$
 Hence, according to Lemma \ref{linearlemma}, 
 \begin{equation}\label{yy}
  v=A_\alpha(1) \Phi +(\alpha+1)(1-\frac\alpha p)\, A_\alpha(1) R \in L^{p^+} \cap
  W^{(\beta+\frac{\alpha-1}2)^-, \infty^-} 
  \end{equation}
   with $ 0<\beta <1-\frac{\alpha}{p} $.
Now, recalling the
  following inequality for any $ s>0 $ (cf. \cite{CW}, Proposition 3.1)
 \begin{equation}
  |D^s (v^{p+1})|_r \lesssim  |v|^p_{r_1}|D^s v |_{r_2}, \quad
 \frac{p}{r_1}+\frac{1}{r_2}=\frac{1}{r} , \quad 1<r,r_1,r_2 < \infty \;,
\label{estim_KPV}
\end{equation}
 we infer that
 $$ v^{p+1}\in W^{(\beta+\frac{\alpha-1}2)^-, 1^+}$$
Finally, Lemma \ref{Strichartzestimates} ensures that
$$
f=A_\alpha(-1) v^{p+1} \in W^{(\beta+\alpha-1)^-,\infty^- } \hookrightarrow C^{0,(\beta+\alpha-1)^-}.
$$

\end{proof}

We are now in a position to prove Proposition \ref{regularitytheorem}.

\begin{proof}[Proof of Proposition \ref{regularitytheorem}] We start by claiming that the fixed point $ v$ belongs to $L^{p^+}\cap W^{(\frac{\alpha-1}2)^-,\infty^-}$. Indeed,  according to the discussion in the beginning of the proof 
of the preceding  lemma, $ v\in L^{p^+}$. 
Now,  writing  again $ v$ as $ v=A_\alpha(1)\Phi+(\alpha+1)(1-\frac\alpha p)\,  A_\alpha(1) R $, we infer from Lemma \ref{linearlemma}   that $ A_\alpha(1) \Phi \in W^{(\frac{\alpha-1}2)^-,\infty^-}$.
 Moreover, according to \eqref{est1R}, $ R\in L^{1^+} $ and  Lemma \ref{Strichartzestimates} then ensures that $ A_\alpha(1)R\in W^{(\frac{\alpha-1}{2})^-,q} $ for $ q$ large enough, which proves the claim.

Now we deduce from \re{estim_KPV} combined with the fact that $v\in L^{p^+}\cap W^{(\frac{\alpha-1}2)^-,\infty^-}$, that $ D^{(\frac{\alpha-1}2)^-}v^{p+1} \in L^{1^+}(\R) $. Hence,  by Lemma \ref{Strichartzestimates} together with   the Sobolev embeddings $ W^{(\frac{\alpha-1}2)^-,\infty^-} \hookrightarrow C^{0,(\frac{\alpha-1}2)^-} $, we infer that $ f=A_\alpha(-1)\frac{v^{p+1}}{p+1} \in C^{0,(\frac{\alpha-1}2)^-} $. An iteration of Lemma \ref{regularitylemma} then yields  $f\in C^{0,\alpha(1-\frac 1p)^-}\hookrightarrow C^1(\R)$ for $p>\alpha/(\alpha-1)$.
\end{proof}
With Proposition \ref{regularitytheorem} in hands, we are in position to complete the proof of Theorem \ref{main-theorem}. This is the aim of the following proposition. 
\begin{pro}
Assume that  $1<\alpha\le 2$   and $p>\alpha/(\alpha-1)$.
Let $ v\in L^{p+1} $ be a fixed point of $ F_\Phi $ where $ \Phi $ is an
homogeneous function of degree $ -\alpha/p $. Then $ u(t,x) =
t^{-\frac{\alpha}{(\alpha+1)p}}
 v(x/t^\frac{1}{\alpha+1}) $ is a solution of the integral equation
 \re{integralKdV}.
\end{pro}
\begin{proof}
 Let $ v $ be a fixed point of  $ F_\Phi $.
Recall that $ R $ given by \eqref{defR}  is well-defined for almost $x\in \R $ since
 $ A_\alpha(-1) g\in C(\R) $ with $g:=\frac{v^{p+1}}{p+1} $.
 Then using  that  $ A_\alpha(1) $ commutes with any function of $ t $
  as well as dilation properties \eqref{tata} of $ A_\alpha $, we infer that
 $ u(t,x) = t^{-\frac{\alpha}{(\alpha+1)p}} v(x/t^\frac{1}{\alpha+1}) $
satisfies for $ t> 0 $,
\begin{align}
 u(t,x) & = 
 t^{-\frac{\alpha}{(\alpha+1)p}} \Bigl(A_\alpha(1) \Phi\Bigr) (\frac{x}{t^\frac{1}{\alpha+1}})
 +(\alpha+1)(1-\frac\alpha p)t^{-\frac{\alpha}{(\alpha+1)p}}
 \Bigl(A_\alpha(1)R(\cdot)\Bigr)\Bigl(\frac{x}{t^\frac{1}{\alpha+1}}\Bigr) \nonumber \\
 & =   t^{-\frac{\alpha}{(\alpha+1)p}} \Bigl(A_\alpha(t) \Phi(\frac{\cdot}{t^\frac{1}{\alpha+1}})\Bigr) (x)
 +(\alpha+1)(1-\frac\alpha p)t^{-\frac{\alpha}{(\alpha+1)p}}\Bigl(A_\alpha(t)R(\frac{\cdot}{t^\frac{1}{\alpha+1}} )\Bigr)(x)
\label{19}
 \; .
\end{align}
 Note that, $ \Phi $ being an homogeneous  function of degree $
-\alpha/p $, it holds 
\begin{equation} \label{20}
t^{-\frac{\alpha}{(\alpha+1)p}} \Bigl( (A_\alpha(t)
\Phi(\frac{\cdot}{t^\frac{1}{\alpha+1}})\Bigr) (x) =(A_\alpha(t) \Phi)(x) \; 
\end{equation}
since $ A_\alpha(t) $ does commute with any function only depending on $t$. 
Moreover, since according to
 Proposition 2, $ A_\alpha(-1) g \in C^1(\R) $, the integral in
the right-hand side of \re{defR}  converges and we are allowed
 to integrate by parts to obtain for $ x\neq 0 $ and $M\ge t^{-1/(\alpha+1)}$,
\begin{align}
& (1-\frac{\alpha}{p})\frac{\sgn(x)}{|x|^{\alpha/p}} \int^{M|x|}_{|x|/t^\frac{1}{\alpha+1}}
\Bigl( (A_\alpha(-1) g)(x/t^{\frac{1}{\alpha+1}})-
  (A_\alpha(-1) g)(\sigma \, \sgn(x))
 \Bigr) \sigma^{-2+\alpha/p} \, d\sigma \nonumber \\
& \quad =
\frac{\sgn(x)}{|x|^\frac{\alpha}{p}} \int^{M|x|}_{|x|/t^{1/(\alpha+1)}}  (A_\alpha(-1) g)'(\sigma \, \sgn(x))  \sigma^{-1+\alpha/p}
 \, d\sigma   \nonumber \\
&\quad\quad +  \frac{\sgn(x)}{M^{1-\alpha/p}\, |x|} \Bigl[ (A_\alpha(-1)g)(x/t^\frac{1}{\alpha+1})
 -(A_\alpha(-1)g)(Mx) \Bigr] \; . \label{1etoile}
\end{align}
Applying that  $ f=A_\alpha(-1)g \in C^{0,\delta} $  with $ \delta=0 $ and $ \delta=0^+ $
we obtain for $ t\ge M^{-(\alpha+1)} $ that 
$$
\Bigl| \frac{\sgn(\cdot)}{M^{1-\alpha/p}|\cdot|} \Bigl[ (A_\alpha(-1)g)(\cdot/t^\frac{1}{\alpha+1})
 -(A_\alpha(-1)g)(M\cdot) \Bigr|_{1^+} \le \frac{1}{M^{(1-\alpha/p)-}} \; .
$$
where we used $ \delta=0 $ to treat the integral over $ |x|\ge 1 $ and $ \delta=0^+ $ to treat the integral over$ |x|<1$.
On the other hand, performing the change of variable
 $ \sigma\mapsto |x|/\tau^{\frac{1}{\alpha+1}} $ and using again dilation
 properties of the group $ A_\alpha $, one infers
\arraycolsep2pt
 \begin{eqnarray}
\frac{\sgn(x)}{|x|^\frac{\alpha}{p}}
\int^{M|x|}_{|x|/t^{1/(\alpha+1)}} & & \hspace*{-3mm} (A_\alpha(-1) g)'(\sigma \sgn(x))  \sigma^{-1+\alpha/p}
 \, d\sigma  \nonumber \\
& = & \int_{1/M^{\alpha+1}}^t \, \tau^{-1-\alpha/(\alpha+1)p}
\Bigl(A_\alpha(-1)g'(\cdot)\Bigr)\Bigl(\frac{x}{\tau^\frac{1}{\alpha+1}}\Bigr)
\, d\tau \nonumber \\
 & = &
  \int_{1/M^{\alpha+1}}^t \, \tau^{-1-\alpha/(\alpha+1)p}
\Bigl(A_\alpha(-\tau)g'\Bigl(\frac{\cdot}{\tau^\frac{1}{\alpha+1}}\Bigr) \Bigr)(x)
\, d\tau \nonumber \\
 & = &
  \int_{1/M^{\alpha+1}}^t  A_\alpha(-\tau)\Bigl(u^p(\tau,\cdot)
u_x(\tau,\cdot )\Bigr)(x) \, d \tau
 \end{eqnarray}
\arraycolsep5pt
Since thanks to \cite{KPV}, Corollary 2.3, one has
\begin{equation}\label{estlinfty}
|A_\alpha(t) \varphi|_\infty \le c t^{-(\frac{1}{\alpha+1})^+} |\varphi|_{1+},
\end{equation}
 applying $ A_\alpha(t) $ to each member of \re{1etoile}, we obtain for
 $ t\ge M^{-(\alpha+1)} $
$$
t^{-\frac{\alpha}{(\alpha+1)p}} \Bigl( A_\alpha(t) R(\frac{\cdot}{t^\frac{1}{\alpha+1}}) \Bigr)(x)
 = \int_{1/M^{\alpha+1}}^t \Bigl( A_\alpha(t-\tau) (u^p(\tau,\cdot) u_x(\tau,\cdot) )\Bigr)(x)
 \, d\tau + \epsilon(1/M) \; ,
$$
where $ \epsilon(y) \to 0 $ as $ y\to 0 $ and we recall that  w. Therefore for any $ (t,x) \in \R_+^*\times
 \R $, the integral of the right hand side of the above equality
converges at the origin  and one has
\begin{equation} \label{2etoile}
t^{-\frac{\alpha}{(\alpha+1)p}} \Bigl( A_\alpha(t) R(\frac{\cdot}{t^\frac{1}{\alpha+1}}) \Bigr)(x)
 = \int_{0}^t \Bigl( A_\alpha(t-\tau) (u^p(\tau,\cdot) u_x(\tau,\cdot) )\Bigr)(x)
 \, d\tau.
\end{equation}
Formulas \re{19}-\re{20} and \re{2etoile} lead to
\begin{equation}
 u(t,x) =(A_\alpha(t)\Phi)(x) - \int_0^t \Bigl[ A_\alpha(t-\tau) (u^p(\tau,\cdot)
u_x(\tau,\cdot))\Bigr](x) \, d\tau , \quad (t,x) \in \R_+^* \times \R. \label{eqint}
\end{equation}
It remains to show that $ u $ satisfies the initial data at least in a
 weak sense. From \re{19}--\re{20}, one has
$$
u(t,x)=(A_\alpha(t)\Phi)(x) -\Bigl(A_\alpha(t)[t^{-\frac{\alpha}{(\alpha+1)p}}
 R(\frac{\cdot}{t^\frac{1}{\alpha+1}} ) ] \Bigr)(x) \; .
$$
Since $ A_\alpha(-1)g\in C^{0,\delta } $, $0\le \delta\le 1 $, it is easy to check that that
\begin{align*}
|t^{-\frac{\alpha}{(\alpha+1)p}} R(\frac{x}{t^\frac{1}{\alpha+1}} )|
 & \leq  \frac{t^{\frac{1}{\alpha+1}(1-\alpha/p)}}{|x|^{\alpha/p}}  \int_{|x|}^\infty
 \sigma^{-2+\alpha/p} |f(\frac{\sigma \sgn(x)}{t^\frac{1}{\alpha+1}})-f(\frac{x}{t^\frac{1}{\alpha+1}}) |  \, d\sigma \\
& \lesssim   \frac{t^{\frac{1}{\alpha+1}(1-\alpha/p)}}{|x|^{\alpha/p}}
 t^{-\delta/(\alpha+1)}
 \int_{|x|}^\infty |\sigma|^{-2 +\delta + \frac{\alpha}{p}} \, d\sigma \\
 & \lesssim   t^{\frac{1}{\alpha+1} (1-\frac{\alpha}{p}-\delta)}
|x|^{-1+\delta},  \quad \forall \,  0\le \delta<1-\frac\alpha p.
\end{align*}
Taking again $ \delta=0 $ to treat the integral over $ |x|\ge 1 $ and $ \delta=0^+ $ for the integral over $ |x|<1 $, we infer that 
$$
\left| t^{-\frac{\alpha}{(\alpha+1)p}} R(\frac{\cdot}{t^\frac{1}{\alpha+1}} ) \right|_{1^+} \lesssim t^{\frac{1}{\alpha+1}(1-\frac \alpha p)^-}.
$$
It follows that for any $\varphi\in \cal{S}(\R)$, one has
\begin{align*}
\left| \left( A_\alpha(t) \left[t^{-\frac{\alpha}{(\alpha+1)p}} R(\frac{\cdot}{t^\frac{1}{\alpha+1}} )\right], \varphi\right)_{L^2}\right| &= \left| \left(t^{-\frac{\alpha}{(\alpha+1)p}} R(\frac{\cdot}{t^\frac{1}{\alpha+1}} ), A_\alpha(-t)\varphi\right)_{L^2}\right| \\
&\lesssim \left| t^{-\frac{\alpha}{(\alpha+1)p}} R(\frac{\cdot}{t^\frac{1}{\alpha+1}} ) \right|_{1^+} |A_\alpha(-t)\varphi|_ {\infty^-}\\
&\lesssim t^{\frac{1}{\alpha+1}(1-\frac \alpha p)^-} |\varphi|_{H^1}.
\end{align*}
Therefore,
$$
A_\alpha(t)[t^{-\frac{\alpha}{(\alpha+1)p}} R(\frac{\cdot}{t^\frac{1}{\alpha+1}} ) ]
 \xrightarrow[t\to 0]{} 0 \quad \mbox{ in } {\cal S}'
$$
and
$$
v(t) -A_\alpha(t) \Phi  \xrightarrow[t\to 0]{} 0 \quad \mbox{ in } {\cal S}'
$$
Finally, since $ \Phi\in {\cal S}' $ and $ A_\alpha(\cdot) $ is continuous on
 $ L^2(\R) $, it is easy to check that
$$
A_\alpha(t) \Phi \xrightarrow[t\to 0]{}\Phi \quad \mbox{ in } {\cal S}'
$$
and thus the initial condition is satisfied at least in $ {\cal S}' $.
\end{proof}

\section{Proof of Theorem \ref{theo2}}
We prove here that the constructed profile $v$ satisfies equation \eqref{profile-eq}. First recall that 
$$
v=A_\alpha(1)\Phi
 +(\alpha+1)(1-\frac\alpha p)\,  A_\alpha(1) R 
 $$
 with $R $ defined in \eqref{defR}. According to  \eqref{yy} in the proof of Proposition \ref{regularitylemma}, if $ f\in C^{0,\beta} $ with $ 0<\beta<1-\frac{\alpha}{p} $ then $v\in W^{(\beta+\frac{\alpha-1}{2})^-,\infty^-} $. For $ p>\frac{\alpha}{\alpha-1} $, Proposition \ref{regularitytheorem} ensures that $ f \in C^{0,1^+} $ and thus we get that $v\in W^{(1-\frac{\alpha}{p} +\frac{\alpha-1}{2})^-,\infty-} $. Hence, $v\in C^1(\R) $ with $v\in L^\infty(\R)  $ and $v'\in L^\infty(\R)  $  whenever 
 $$
 1-\frac{\alpha}{p} + \frac{\alpha-1}{2}>1 \Leftrightarrow p > \frac{2\alpha}{\alpha-1} \; .
 $$
Therefore, for $  p>\frac{2\alpha}{\alpha-1}$, differentiating \eqref{eqint} with respect to $ t$ and recalling that $ u(t,x) = t^{-\frac{\alpha}{(\alpha+1)p}} v(x/t^\frac{1}{\alpha+1}) $, we obtain for $ t=1 $, 
 \begin{align}
  - \frac{\alpha}{(\alpha+1)p}v &- \frac{x}{\alpha+1} v'= \partial_{t_{\vert t=1}} u  \nonumber\\
  =& \lim_{\tau\to 0} \frac{u(1+\tau)-u(1)}{\tau} \nonumber\\
 = &\lim_{\tau\to 0} \Bigl[ \frac{(A_\alpha(1+\tau) -A_\alpha(1))\Phi}{\tau}\nonumber\\
&- \frac{A_\alpha(1+\tau)\int_0^{1+\tau}A_\alpha(-s) \partial_x (u^{p+1}(s)) \, ds -A_\alpha(1)\int_0^{1} A_\alpha(-s) \partial_x (u^{p+1}(s)) \,ds }{(p+1) \; \tau}\Bigr] \label{deriv}
 \end{align}
 Now, according to Lemma \ref{linearlemma}, $ A_\alpha(t) \Phi \in L^\infty(\R) $ for $ t\sim 1 $ and thus for any $ \varphi\in {\mathcal S}(\R) $ it holds
  \begin{align}
 \lim_{\tau\to 0} \Bigl( \frac{(A_\alpha(1+\tau)-A_\alpha(1)) \Phi}{\tau},  \varphi \Bigr)_{\S',\S} & = 
 \lim_{\tau\to 0} \Bigl( \Phi, \frac{(A_\alpha(-1-\tau)-A_\alpha(-1))  \varphi}{\tau} \Bigr)_{\S',\S}  \nonumber\\
  &= \Bigl( \Phi,  \partial_{t_{\vert t=1}}(A_\alpha(-t) \varphi) \Bigr)_{\S',\S}\nonumber \\
  &=\Bigl( \Phi, - D^\alpha \partial_x A_\alpha(-1) \varphi\Bigr)_{\S',\S} \nonumber\\
  & = \Bigl(  D^\alpha \partial_x A_\alpha(1)\Phi,  \varphi\Bigr)_{\S',\S}\label{deriv2}
 \end{align}
To treat  the second term in the right-hand side of \eqref{deriv}, we first notice that 
\begin{align}
A_\alpha(1+\tau)\int_0^{1+\tau}A_\alpha(-s)  \partial_x  (u^{p+1}(s))  \, ds -&A_\alpha(1)\int_0^{1} A_\alpha(-s) \partial_x  (u^{p+1}(s)) \,ds \nonumber \\
&= (A_\alpha(1+\tau)-A_\alpha(1))\int_0^{1} A_\alpha(-s) \partial_x  (u^{p+1}(s)) \,ds\nonumber\\
& + A_\alpha(1+\tau)\int_1^{1+\tau} A_\alpha(-s) \partial_x  (u^{p+1}(s)) \,ds \nonumber\\
&=B_1+B_2 \; .\label{deriv3}
\end{align}
According to the proof of Proposition \ref{regularitytheorem}, the fixed point $ v$ belongs to $ L^{p^+} \cap L^\infty $
 and the expression of $u$ leads to 
  $$
 |u^{p+1}(s, \cdot)|_2 = s^{\frac{p-2\alpha(p+1)}{2(\alpha+1)p}} |v|^{p+1}_{2(p+1)}\; .
 $$
Since $A_\alpha$ is unitary in $ L^2 $ and $\frac{p-2\alpha(p+1)}{2(\alpha+1)p}>-1 \Leftrightarrow 3p>2\alpha $, we infer that 
$ \int_0^t A_\alpha(-s) u^{p+1}(s)\, ds \in L^2 $. 
Since $ \partial_x $ commutes with the integral in $s $ and with the linear evolution $ A_\alpha $, we thus may write 
 \begin{align}
 \lim_{\tau\to 0} \Bigl(\frac{B_1}{\tau} ,  \varphi \Bigr)_{\S',\S} & = 
 \lim_{\tau\to 0} \Bigl( \int_0^{1} A_\alpha(-s)  u^{p+1}(s) \,ds,  \frac{-(A_\alpha(-1-\tau)-A_\alpha(-1))  \varphi'}{\tau}\Bigr)_{\S',\S} \nonumber \\
  &= \Bigl( \int_0^{1} A_\alpha(-s)  u^{p+1}(s) \,ds,  -\partial_{t_{\vert t=1}}(A_\alpha(-t) \varphi ') \Bigr)_{\S',\S} \nonumber\\
  &=\Bigl( \int_0^{1} A_\alpha(-s)  u^{p+1}(s) \,ds, D^\alpha \partial_x A_\alpha(-1) \varphi'\Bigr)_{\S',\S} \nonumber\\
  & = \Bigl(  D^\alpha \partial_x A_\alpha(1)\int_0^{1} A_\alpha(-s) \partial_x  u^{p+1}(s) \,ds,  \varphi\Bigr)_{\S',\S}
  \label{deriv4}
 \end{align}
Finally, since  $s\mapsto u^{p+1}(s) $ is continuous with values in $ L^2(\R) $ for $ s$ close to $1$, we may also write 
 \begin{align}
 \lim_{\tau\to 0} \Bigl(\frac{B_2}{\tau} ,   \varphi \Bigr)_{\S',\S} & = 
 \lim_{\tau\to 0} \Bigl( \frac{\int_1^{1+\tau} A_\alpha(-s)  u^{p+1}(s) \,ds}{\tau} , -A_\alpha(-1-\tau)  \varphi'\Bigr)_{\S',\S}  \nonumber\\
  &= \Bigl(  A_\alpha(-1) u^{p+1}(1) , -A_\alpha(-1)  \varphi' \Bigr)_{\S',\S}\nonumber \\
  &= \Bigl(  \partial_x (u^{p+1})(1) ,  \varphi \Bigr)_{\S',\S} \label{deriv5}
  \end{align}
Hence, gathering \eqref{deriv}-\eqref{deriv5} and using \eqref{eqint} with $t=1$ we infer that the following equality holds in  distributional sense 
$$
- \frac{\alpha}{(\alpha+1)p}v - \frac{x}{\alpha+1} v'=-\frac{1}{p+1}\partial_x (u^{p+1}) (1) +D^\alpha \partial_x u(1) 
$$
which gives \eqref{profile-eq}.  In particular
\begin{equation}\label{reg}
D^\alpha v'=- \frac{\alpha}{(\alpha+1)p}v - \frac{x v'}{\alpha+1} +\frac{1}{p+1} (v^{p+1})' \in L^\infty_{loc}(\R).
\end{equation}
Hence, $ v'\in W^{\alpha, \infty}_{loc} $ which ensures that $ v\in C^2(\R) $ since $ \alpha>1$. Re-injecting this information in \eqref{reg}, we get that actually $ v\in C^3(\R) $  and repeating this argument we obtain by direct recurrence that actually $ v\in C^\infty(\R) $.

 \end{document}